\documentclass[a4paper,12pt,reqno]{amsart}

 \usepackage{amssymb}

\usepackage[text={130mm,200mm}]{geometry}

\theoremstyle{plain}
\newtheorem{theorem}{Theorem}
\newtheorem*{theorem*}{Theorem}

\newtheorem*{corollary*}{Corollary}

\newtheorem*{lemma*}{Lemma}
\newtheorem{proposition}{Proposition}
\newtheorem*{proposition*}{Proposition}

\newtheorem*{conjecture*}{Conjecture}
\theoremstyle{definition}

\newtheorem*{definition*}{Definition}
\theoremstyle{remark}

\newtheorem*{remark*}{Remark}

\begin{document}

\title[Fractal properties of the Minkowski function]{On one fractal property of the Minkowski function}
\author{Symon Serbenyuk}
\address{Institute of Mathematics \\
 National Academy of Sciences of Ukraine \\
  3~Tereschenkivska St. \\
  Kyiv  \\
  01004 \\
  Ukraine}
\email{simon6@ukr.net;  simon.mathscience@imath.kiev.ua}

\subjclass[2010]{28A80,  26A30, 11K55,  03E99, 11K50}

\keywords{Function with complicated local structure, singular function, fractal, self-similar set, Hausdorff-Besicovitch dimension, Minkowski function.}

\begin{abstract}

The article is devoted to answer the question about preserving the Hausdorff-Besicovitch dimension by the singular Minkowski function. It is proved that the function is not the DP-transformation, i.e., the Minkowski function does not preserve the   Hausdorff-Besicovitch dimension.
 
\end{abstract}

\maketitle

It is well-known that the main problem of the fractals theory is a problem of calculating of dimension. In particular, the dimension is the Hausdorff-Besicovitch dimension (a fractal  dimension). But, since some classes of sets have complicated determination, a calculation of a value of the dimension is a difficult and labour-consuming problem for these sets. Therefore, to simplify of fractal dimension calculation, a problem of searching of auxiliary facilities appears. Transformations preserving the Hausdorff-Besicovitch dimension (DP-transformations) are such facilities. The transformations help to simplify of sets determinations and to investigate of belonging to the class of DP-tranformations of other transformations.

 Monotone singular distribution functions as transformations of the segment $[0;1]$ are very interesting  for studying of DP-transformations. The present article is devoted to considering of fractal properties of one example of  such functions. 
In the article a preserving of  the Hausdorff-Besicovitch dimension by the Minkowski function is investigated.  Results of the present article  were presented by the author of the article on Second Interuniversity Scientific Conference on Mathematics and Physics for Young Scientists in April, 2011 \cite{S. Serbenyuk abstract 1}.

The following function
$$
G(x)=2^{1-a_1}-2^{1-(a_1+a_2)}+...+(-1)^{n-1}2^{1-(a_1+a_2+...+a_n)}+...,
$$
is called \emph{the Minkowski function}. An argument $x$ of the function  determined in terms of representation of real numbers by continued fractions, i. e., 
$$
  [0;1]\ni x=  \cfrac{1}{a_1 + \cfrac{1}{a_2 +  \cfrac{1}{a_3 + \ldots}}}\equiv [0; a_1, a_2,...,a_n,...], ~a_n\in\mathbb N.
$$
 To establish of the one-to-one correspondence between rational numbers and quadratic irrationalities,  the function was introduced by Minkowski in \cite{{Minkowski1911}}. A difficulty of  investigation of  preserving the Hausdorff-Besicovitch dimension by the Minkowski function is caused by singularity and a complexity of  the argument determination of the function.

Let us consider the following set 
$$
E_9\equiv\left\{ x: x=[0;a_1,a_2,...,a_i,...], a_i \in\{1, 2, 3, ..., 9\}~  \forall{ i\in \mathbb N}\right\}.
$$

"The geometry of continued fractions" does not have a property of "classical self-similarity". It is the reason of rather greater difficulties for  obtaining  exact results  of fractal properties of continued fractions sets.  


The main researches, in which  fractal properies of some types of sets of continued fractions are studied, there are the articles of Jarnik \cite{Jarnik1929} (the sets $E_n$, that  representation of its elements by continued fractions  contains symbols that do not exceed $n$), Good \cite{Good1941} (the sets of continued fractions, whose  elements $a_n(x)$ quickly tend to infinity), Hirst \cite{Hirst1970} (the sets of such continued fractions, whose  elements belong to infinite sequence of positive integers and increase indefinitely) and in   \cite{Hensley1992} Hensley  clarified estimations for Hausdorff-Besicovitch dimension of $E_n$ such that were obtained by Jarnik and Good:
$$
1-\frac{1}{n\lg{2}}\le\alpha_0(E_n)\le1-\frac{1}{8n\lg{n}}, n>8.
$$

Whence, for $ E_9$  we obtain that
\begin{equation}
\label{eq: estimations1}
0,6308969 \le\alpha_0(E_9)\le 0,985445112.
\end{equation}

The   answer to the main question of this article follows from the next statement about a value of the Hausdorff-Besicovitch dimension of the set $G(E_9)$. 
\begin{theorem}
  \label{thm}
A value of the Hausdorff-Besicovitch dimension of the following set
$$
\left\{y: y=G(x)=\frac{2}{2^{a_1}}-\frac{2}{2^{a_1+a_2}}+...+\frac{2(-1)^{n-1}}{2^{a_1+a_2+...+a_n}}+..., x=[0;a_1,a_2,...,a_n,...]\right\}, 
$$
where $\forall{i \in \mathbb N}:   a_i \in \{{ k_1}, {k_2}, ..., {k_S} \}$   and $\{{ k_1}, {k_2}, ..., {k_S} \}$ is a fixed tuple of positive integers for a fixed positive integer number  $S>1$, 
can be calculate by the formula: 
$$
\left({\frac{1}{2^{k_1}}}\right)^{\alpha_0}+\left({\frac{1}{2^{k_2}}}\right)^{\alpha_0}+\left({\frac{1}{2^{k_3}}}\right)^{\alpha_0}+...+\left({\frac{1}{2^{k_S}}}\right)^{\alpha_0}=1.
$$
\end{theorem}
\begin{proof}
Let us write by $ \Delta_0=$
$$
 =\left\{y: y=\frac{2}{2^{a_1}}-\frac{2}{2^{a_1+a_2}}+...+\frac{2(-1)^{n-1}}{2^{a_1+a_2+...+a_n}}+...+\frac{2}{2^{a_1+a_2+...+a_n}}\left(\frac{(-1)^{n}}{2^{a_{n+1}}}+\frac{(-1)^{n+1}}{2^{a_{n+1}+a_{n+2}}}+...\right)\right\} 
$$
the set 
$$
 \Delta_0\equiv\left\{y: y=G(x)=\frac{2}{2^{a_1}}-\frac{2}{2^{a_1+a_2}}+...+\frac{2(-1)^{n-1}}{2^{a_1+a_2+...+a_n}}+...\right\}.
$$
Call the set 
$$
D_n\equiv\left\{y: y=\frac{(-1)^{n}}{2^{a_{n+1}}}+\frac{(-1)^{n+1}}{2^{a_{n+1}+a_{n+2}}}+\frac{(-1)^{n+2}}{2^{a_{n+1}+a_{n+2}+a_{n+3}}}+...\right\}
$$
by \emph{an "indicative set" of rank $n$}.

Let $S=2$ and  $k_1<k_2$, then  $\frac{2}{2^{k_1}}>\frac{2}{2^{k_2}}$, 
  $$
\sup\Delta_0={\frac{2}{2^{k_1}}-{\frac{2}{2^{k_1+k_2}}}}+{\frac{2}{2^{k_1+k_2+k_1}}}-{\frac{2}{2^{k_1+k_2+k_1+k_2}}}+...=
$$
$$
=~\frac{{\frac{2}{2^{k_1}}}}{1-\frac{1}{2^{k_1+k_2}}}-\frac{{\frac{2}{2^{{k_1+k_2}}}}}{1-\frac{1}{2^{k_1+k_2}}}=\frac{2(2^{k_2}-1)}{2^{k_1+k_2}-1},
$$
$$
\inf\Delta_0=\{{\frac{2}{2^{k_2}}-{\frac{2}{2^{k_2+k_1}}}}+{\frac{2}{2^{k_2+k_1+k_2}}}-{\frac{2}{2^{k_2+k_1+k_2+k_1}}}+...=
$$
$$
=~\frac{{\frac{2}{2^{k_2}}}}{1-\frac{1}{2^{k_1+k_2}}}-\frac{{\frac{2}{2^{{k_1+k_2}}}}}{1-\frac{1}{2^{k_1+k_2}}}=\frac{2(2^{k_1}-1)}{2^{k_1+k_2}-1}.
$$
So, 
$$
\lambda(\Delta_0)=\frac{2\left(2^{k_2}-2^{k_1}\right)}{2^{k_1+k_2}-1}, 
$$
 where $\lambda(\cdot)$  is a diameter  of a set. 

Consider the following two sets
$$
  \Delta_{k_{j_1}}=\left\{y: y= \frac{2}{2^{k_{j_1}}}-\frac{2}{2^{k_{j_1}+a_2}}+...+\frac{2(-1)^{n-1}}{2^{k_{j_1}+a_2+...+a_n}}+...\right\} 
$$
where $j_1\in\{1,2\}$ and $\{k_1,k_2\}\ni$ $k_{j_1}$ is a fixed for $\Delta_{k_{j_1}}$.
That is
$$
\Delta_{k_1}=\left\{y: y={\frac{2}{2^{k_1}}\left(1-\frac{1}{2^{a_2}}+\frac{1}{2^{a_2+a_3}}-\frac{1}{2^{a_2+a_3+a_4}}+...\right)}\right\}=\left\{\frac{2}{2^{k_1}}-{\frac{2}{2^{k_1}}}{D_1}\right\},
$$
$$
\Delta_{k_2}=\left\{y: y={\frac{2}{2^{k_2}}\left(1-\frac{1}{2^{a_2}}+\frac{1}{2^{a_2+a_3}}-\frac{1}{2^{a_2+a_3+a_4}}+...\right)}\right\}=\left\{\frac{2}{2^{k_2}}-{\frac{2}{2^{k_2}}}{D_1}\right\}.
$$

It is easy to see that
$$
\lambda(\Delta_{k_1})={\frac{2}{2^{k_1}}}\lambda(D_1)={\frac{1}{2^{k_1}}}\lambda(\Delta_0) 
$$
and
$$
 \lambda(\Delta_{k_2})={\frac{2}{2^{k_2}}}\lambda(D_1)={\frac{1}{2^{k_2}}}\lambda(\Delta_0).
$$

In the second step we obtain the following  four sets:  $\Delta_{k_1k_1}, \Delta_{k_1k_2}, \Delta_{k_2k_1}, \Delta_{k_2k_2}$.
In the $n$th step we shall have $2^n$ sets 
$$
\Delta_{k_{j_1}k_{j_2}...k_{j_n}}\equiv \left\{y: y=\sum^{n} _{m=1}{\frac{2(-1)^{m-1}}{2^{k_{j_1}+k_{j_2}+...+k_{j_m}}}}+ \sum^{\infty} _{t=n+1}{\frac{2(-1)^{t-1}}{2^{k_{j_1}+k_{j_2}+...+k_{j_n}+a_{n+1}+...+a_t}}}\right\},
$$ 
where $k_{j_1}, k_{j_2},..., k_{j_n}$ is a fixed tuple of numbers from $\{k_1,k_2\}$, and the following expression is true for all $n\in\mathbb N$.
$$
\frac{\lambda\left(\Delta_{k_{j_1}k_{j_2}...k_{j_{n-1}}k_{j_n}}\right)}{\lambda\left(\Delta_{k_{j_1}k_{j_2}...k_{j_{n-1}}}\right)}
=\frac{1}{2^{k_{j_n}}}\in\left\{\frac{1}{2^{k_1}},\frac{1}{2^{k_2}}\right\},
$$

The last-mentioned fact follows from the next proposition.
\begin{proposition}
  \label{l1}
 The condition 
$$
\frac{\lambda(D_{n+1})}{\lambda(D_{n})}=1
$$
holds for all $n\in \mathbb N$.
\end{proposition}
Really, let  $n=2l$, $ l \in \mathbb N$. Then
$$
D_{n+1}=D_{2l+1}\equiv\left\{y: y=-\frac{1}{2^{a_{2l+2}}}+\frac{1}{2^{a_{2l+2}+a_{2l+3}}}-\frac{1}{2^{a_{2l+2}+a_{2l+3}+a_{2l+4}}}+...\right\},
$$
$$
D_{n}=D_{2l}\equiv\left\{y: y=\frac{1}{2^{a_{2l+1}}}-\frac{1}{2^{a_{2l+1}+a_{2l+2}}}+\frac{1}{2^{a_{2l+1}+a_{2l+2}+a_{2l+3}}}-\frac{1}{2^{a_{2l+1}+a_{2l+2}+a_{2l+3}+a_{2l+4}}}+...\right\}.
$$
$$
\sup(D_{2l+1})=-\frac{\frac{1}{2^{k_2}}}{1-\frac{1}{2^{k_1+k_2}}}+\frac{\frac{1}{2^{k_1+k_2}}}{1-\frac{1}{2^{k_1+k_2}}}=\frac{1-2^{k_1}}{2^{k_1+k_2}-1},
$$
$$
\inf(D_{2l+1})=-\frac{\frac{1}{2^{k_1}}}{1-\frac{1}{2^{k_1+k_2}}}+\frac{\frac{1}{2^{k_1+k_2}}}{1-\frac{1}{2^{k_1+k_2}}}=\frac{1-2^{k_2}}{2^{k_1+k_2}-1},
$$
$$
\lambda(D_{2l+1})=\frac{2^{k_2}-2^{k_1}}{2^{k_1+k_2}-1}.
$$
Similarly,
$$
\sup(D_{2l})=\frac{\frac{1}{2^{k_1}}}{1-\frac{1}{2^{k_1+k_2}}}-\frac{\frac{1}{2^{k_1+k_2}}}{1-\frac{1}{2^{k_1+k_2}}}=\frac{2^{k_2}-1}{2^{k_1+k_2}-1},
$$
$$
\inf(D_{2l})=\frac{\frac{1}{2^{k_2}}}{1-\frac{1}{2^{k_1+k_2}}}-\frac{\frac{1}{2^{k_1+k_2}}}{1-\frac{1}{2^{k_1+k_2}}}=\frac{2^{k_1}-1}{2^{k_1+k_2}-1},
$$
$$
\lambda(D_{2l})=\frac{2^{k_2}-2^{k_1}}{2^{k_1+k_2}-1}.
$$

So, $\lambda(D_{2l+1})=\lambda(D_{2l})$. A proof of the equality is analogical for the case of  $n=2l+1$.

Let $I_{k_{j_1}k_{j_2}...k_{j_{n-1}}k_{j_n}}$ be a segment, whose endpoints coincide with endpoints of the corresponding set $\Delta_{k_{j_1}k_{j_2}...k_{j_{n-1}}k_{j_n}}$.
Since  $\Delta_0\subset I_0$ and $\Delta_0$ is a perfect set and 
$$
\Delta_0=[I_{k_1}\cap\Delta_0]\cup[I_{k_2}\cap\Delta_0], 
$$
$$
[I_{k_1}\cap\Delta_0]\stackrel{2^{-k_1}}{\thicksim}\Delta_0, ~~~[I_{k_2}\cap\Delta_0]\stackrel{2^{-k_2}}{\thicksim}\Delta_0,
$$
the theorem proved for the case of $S=2$.

Let $S>2$ and $ k_1<k_2<...<k_S$.  Similarly, in the $n$th step we shall have the  $S^n$ sets $\Delta_{k_{j_1}k_{j_2}...k_{j_n}}$.
It is easy to see that 
$$
\frac{\lambda(\Delta_{k_{j_1}k_{j_2}...k_{j_{n-1}}k_{j_n}})}{\lambda(\Delta_{k_{j_1}k_{j_2}...k_{j_{n-1}}})}=\frac{1}{2^{k_{j_n}}},
$$
where $k_{j_n} \in \{k_1, k_2,..., k_S\}, S\in \mathbb N, S>1$. 

Since the set $\Delta_0$ is a compact self-similar set of the space $\mathbb R^1$, the Hausdorff-Besicovitch dimension $\alpha_0(\Delta_0)$ of the set $\Delta_0$ is calculating by the formula:
$$
\left({\frac{1}{2^{k_1}}}\right)^{\alpha_0}+\left({\frac{1}{2^{k_2}}}\right)^{\alpha_0}+\left({\frac{1}{2^{k_3}}}\right)^{\alpha_0}+...+\left({\frac{1}{2^{k_S}}}\right)^{\alpha_0}=1.
$$

The theorem is proved. \end{proof}

So, for $k_1=1, k_2=2, ..., k_9=9 $  we obtain that 
\begin{equation}
\label{eq: formula}
\alpha_0(G(E_9))\approx 0,9985778625536. 
\end{equation}

From \eqref{eq: estimations1} and \eqref{eq: formula} it follows that $\alpha_0(E_9)\ne \alpha_0(G(E_9))$. 
So, the Minkowski function does not preserve the Hausdorff-Besicovitch dimension.


\begin{thebibliography}{9}

\bibitem{Good1941} G. T. Good, The fractional dimensional theory of continued fractions,  \emph{Proc. Cambridge Phil. Soc.}    \textbf{37}, 199--228 (1941).  

\bibitem{Hensley1992} D. Hensley, Continued fraction Cantor sets, Hausdorff dimension and functional analysis,  \emph{Journal of Number Theory}    \textbf{40}, 336--358 (1992).  

\bibitem{Hirst1970} K. E. Hirst, Fractional dimension theory of continued fractions,  \emph{Quart. J. Math.} \textbf{21}, 29--35 (1970).  

\bibitem{Jarnik1929} V. Jarnik,  Zur metrichen Theorie den diophantischen Approximationen,  \emph{Recueil Math. Moscow}    \textbf{36}, 371--382 (1929).  

\bibitem{Minkowski1911} H. Minkowski,  Gesammeine Abhandlungen,  \emph{Berlin}    \textbf{2}, 50--51 (1911).  




\bibitem{S. Serbenyuk abstract 1} S. O. Serbenyuk, Preserving of Hausdorff-Besicovitch dimension by the monotone singular distribution functions, \emph{Second Interuniversity Scientific Conference on Mathematics and Physics for Young Scientists: Abstracts,}  Kyiv, 2011, P.~106-107. (in Ukrainian) Link:  https://www.researchgate.net/publication/301637057


\end{thebibliography}
\end{document}